\documentclass[a4paper,12pt]{article}

\usepackage{color,makeidx,bm,latexsym,amscd,amsmath,amssymb,stmaryrd,amsthm,float,graphicx,psfrag}

\theoremstyle{plain}
\newtheorem{theorem}{Theorem}[section]

\theoremstyle{definition}

\newtheorem{remark}[theorem]{Remark}

\newcommand{\scL}{\ensuremath{\mathcal{L}}}

\newcommand{\A}{\mathcal{A}}
\newcommand{\C}{\mathbb{C}}

\newcommand{\R}{\mathbb{R}}
\newcommand{\Z}{\mathbb{Z}}

\title{Eulerian polynomials and polynomial congruences}
\author{
Kazuki Iijima\thanks{Department of Mathematics, 
Hokkaido University, North 10, West 8, Kita-ku, 
Sapporo 060-0810, JAPAN 
E-mail: k.iijima1203@gmail.com}, 
Kyouhei Sasaki\thanks{Department of Mathematics, 
Hokkaido University, North 10, West 8, Kita-ku, 
Sapporo 060-0810, JAPAN 
E-mail: kyo.sasaki2308@gmail.com}, \\
Yuuki Takahashi\thanks{Department of Mathematics, 
Hokkaido University, North 10, West 8, Kita-ku, 
Sapporo 060-0810, JAPAN 
E-mail: chldmrhnmzk\_rmhodq@yahoo.co.jp
}, 
Masahiko Yoshinaga\thanks{Department of Mathematics, 
Hokkaido University, North 10, West 8, Kita-ku, 
Sapporo 060-0810, JAPAN 
E-mail: yoshinaga@math.sci.hokudai.ac.jp}}
\date{\today}
\begin{document}
\maketitle

\begin{abstract} 
We prove that the Eulerian polynomial satisfies certain polynomial 
congruences. Furthermore, these congruences characterize the Eulerian 
polynomial. 

\end{abstract}



\section{Introduction}
\label{sec:intro}

The Eulerian polynomial $A_\ell(x)$ ($\ell\geq 1$) was introduced by Euler 
in the study of sums of powers \cite{foa-hist}. 
In this paper, we define the Eulerian polynomial $A_\ell(x)$ 
as the numerator of the rational function 
\begin{equation}
\label{eq:euleriandef}
F_\ell(x)=
\sum_{k=1}^\infty k^\ell x^k=\left(x\frac{d}{dx}\right)^\ell \frac{1}{1-x}
=
\frac{A_\ell (x)}{(1-x)^{\ell +1}}. 
\end{equation}
The first few examples are $A_1(x)=x, A_2(x)=x+x^2, A_3(x)=x+4x^2+x^3, 
A_4(x)=x+11 x^2+11 x^3+x^4$, etc. 
The Eulerian polynomial $A_\ell (x)$ is 
a monic of degree $\ell$ with positive integer coefficients. Write $A_\ell (x)=
\sum_{k=1}^\ell A(\ell , k)x^k$. The coefficient $A(\ell , k)$ is called 
an Eulerian number. 

In 1950s, Riordan \cite{rio-com} discovered a combinatorial interpretation 
of Eulerian numbers in terms of descents and ascents of permutations, and 
Carlitz \cite{car-q} defined $q$-Eulerian numbers. Since 
then Eulerian numbers are actively studied in enumerative combinatorics. 
(See \cite{comtet, st-ec1, pe}.)

Another combinatorial application of the Eulerian polynomial 
was found in the theory of hyperplane arrangements 
\cite{yos-worp, yos-excep}. 
The characteristic polynomial of 
the so-called Linial arrangement \cite{ps-def} can be expressed in terms of 
the root system generalization of Eulerian polynomials introduced by 
Lam and Postnikov \cite{lp-alc2}. 
The comparison of expressions in \cite{ps-def} and \cite{yos-worp} 
yields the following. 

\begin{theorem}
\label{thm:main1intro}
(\cite[Proposition 5.5]{yos-worp})
For $\ell , m\geq 2$, the Eulerian polynomial $A_\ell (x)$ 
satisfies the following 
\begin{equation}
\label{eq:congru}
A_\ell (x^m)\equiv
\left(
\frac{1+x+x^2+\dots+x^{m-1}}{m}
\right)^{\ell +1}
A_\ell (x) 
\mod\ (x-1)^{\ell +1}. 
\end{equation}
\end{theorem}


The purpose of this paper is two-fold. 
First, we give a direct and 
simpler proof of Theorem \ref{thm:main1intro}. 
Second, we prove the converse of the above theorem. 
Namely, the congruence (\ref{eq:congru}) characterizes the Eulerian 
polynomial as follows. 

\begin{theorem}
\label{thm:main2intro}
Let $f(x)$ be a monic of degree $\ell$. Then, $f(x)=A_\ell(x)$ if and only if 
the congruence (\ref{eq:congru}) holds for some $m\geq 2$. 
(See Theorem \ref{thm:main2proof}.) 
\end{theorem}

The remainder of this paper is organized as follows. 
After recalling classical results on Eulerian polynomials in 
\S \ref{sec:review}, 
we briefly describe 
in \S \ref{sec:background} 
the proof of the congruence (\ref{eq:congru}) in \cite{yos-worp} that 
is based on the expression of 
characteristic polynomials of Linial hyperplane arrangements. 
In \S \ref{sec:direct}, we give a direct proof of the congruence. 
In \S \ref{sec:char}, we give the proof of Theorem \ref{thm:main2intro}. 

\begin{remark}
The right-hand side of (\ref{eq:congru}) is discussed also in 
\cite[Proposition 2.5]{st-za}. 
\end{remark}

\section{Brief review of Eulerian polynomials}
\label{sec:review}

In this section, we recall classical results on the Eulerian polynomial 
and the Eulerian numbers $A(\ell , k)$. 
By definition (\ref{eq:euleriandef}), 
the Eulerian polynomial $A_\ell (x)$ satisfies the relation 
\[
\frac{A_{\ell}(x)}{(1-x)^{\ell +1}}=
x\frac{d}{dx}
\frac{A_{\ell -1}(x)}{(1-x)^{\ell }}, 
\]
which yields the following recursive relation. 
\begin{equation}
\label{eq:recursive}
A(\ell , k)=k\cdot A(\ell -1, k)+(\ell -k+1)\cdot A(\ell -1, k-1). 
\end{equation}

Consider the coordinate change $w=\frac{1}{x}$. Then, the Euler 
operator is transformed as 
$x\frac{d}{dx}=-w\frac{d}{dw}$. The direct computation using the relation 
$\frac{1}{1-\frac{1}{w}}=1-\frac{1}{1-w}$ yields 
$x^{\ell +1}A_\ell (\frac{1}{x})=A_\ell (x)$. 
Equivalently, $A(\ell , k)=A(\ell , \ell +1-k)$. 

Definition (\ref{eq:euleriandef}) is also equivalent to 
\begin{equation}
\label{eq:poly}
A_\ell (x)=(1-x)^{\ell +1}
\cdot\sum_{k=0}^\infty k^\ell x^k
\end{equation}
and 
\begin{equation}
\label{eq:power}
\sum_{k=0}^\infty k^\ell x^k 
=
A_\ell(x)\cdot 
\sum_{k=0}^{\infty}(-x)^k
\binom{-\ell -1}{k}. 
\end{equation}
Then, (\ref{eq:poly}) yields 
\begin{equation}
\label{eq:eulerianbinom}
A(\ell , k)=\sum_{j=0}^{k}(-1)^j\binom{\ell +1}{j}(k-j)^\ell , 
\end{equation}
and (\ref{eq:power}) yields 
\begin{equation}
\label{eq:numworp}
k^\ell =\sum_{j=1}^\ell A(\ell , j)\binom{k+\ell -j}{\ell }. 
\end{equation}
Note that both sides of (\ref{eq:numworp}) are polynomials of 
degree $\ell $ in $k$, and it holds for any $k>0$. Hence the equality 
holds at the level of polynomials in $t$. Thus, we have 
\begin{equation}
\label{eq:polyworp}
t^\ell =\sum_{j=1}^\ell A(\ell , j)\binom{t+\ell -j}{\ell }, 
\end{equation}
which is called the Worpitzky identity \cite{wor}. 
Using the shift operator $S:t\longmapsto t-1$ 
(see \S \ref{sec:background}), 
the Worpitzky identity can be written as 
\begin{equation}
\label{eq:worp}
t^\ell =A_\ell (S)
\binom{t+\ell }{\ell }. 
\end{equation}
Next, we consider 
exponential generating series of $A_\ell(x)$, and describe relations 
with Bernoulli numbers. First, using (\ref{eq:poly}), we have 
\begin{equation}
\label{eq:expgen}
\begin{split}
\sum_{\ell=0}^\infty
\frac{A_\ell(x)}{\ell !}t^\ell
&=
\sum_{\ell=0}^\infty
\frac{t^\ell}{\ell !}
(1-x)^{\ell+1}\sum_{n=0}^\infty n^\ell x^n
\\
&=
(1-x)\sum_{n=0}^\infty x^ne^{nt(1-x)}\\
&=
\frac{1-x}{1-xe^{t(1-x)}}. 
\end{split}
\end{equation}
Replacing $x$ by $-1$ in (\ref{eq:expgen}), we have 
\begin{equation}
\label{eq:A-1}
\sum_{\ell=0}^\infty
\frac{A_\ell(-1)}{\ell !}t^\ell
=
\frac{2}{1+e^{2t}}. 
\end{equation}
Recall that the Bernoulli polynomial $B_\ell(x)$ is defined by 
\begin{equation}
\label{eq:Bpoly}
\sum_{\ell=0}^\infty
\frac{B_\ell(x)}{\ell !}t^\ell
=\frac{te^{xt}}{e^t-1}, 
\end{equation}
($B_0(x)=1, B_1(x)=x-\frac{1}{2}, B_2(x)=x^2-x+\frac{1}{6}, 
B_3(x)=x^3-\frac{3}{2}x^2+\frac{1}{2}x, 
B_4(x)=x^4-2x^3+x^2-\frac{1}{30}, \cdots$) 
and the constant term $B_\ell(0)$ is called the Bernoulli number. 
Replacing $x$ by $0$ and $t$ by $at$ with $a\in\mathbb{C}$ 
in (\ref{eq:Bpoly}), we have 
\begin{equation}
\label{eq:B0}
\sum_{\ell=0}^\infty
\frac{B_\ell(0)}{\ell !}(at)^\ell
=\frac{at}{e^{at}-1}. 
\end{equation}
Using the identity 
$\frac{2t}{e^{2t}+1}=\frac{2t}{e^{2t}-1}-\frac{4t}{e^{4t}-1}$, 
the Bernoulli number $B_\ell(0)$ can be expressed as 
\begin{equation}
\label{eq:Bnumber}
B_\ell(0)=\frac{\ell}{2^\ell(1-2^\ell)} A_{\ell-1}(-1). 
\end{equation}
There is another relation between Eulerian polynomials and Bernoulli 
polynomials. Let $\ell>0$. 
Using (\ref{eq:numworp}) and the famous formula 
$\sum_{x=0}^{N-1}x^\ell=\frac{B_{\ell+1}(N)-B_{\ell+1}(0)}{\ell+1}$, 
we have 
\begin{equation}
\label{eq:B2nd}
B_{\ell+1}(N)-B_{\ell+1}(0)=
(\ell+1)\cdot\sum_{k=1}^\ell A(\ell, k)
\binom{\ell+N-k}{\ell+1}.
\end{equation}
With the shift operator $S$, (\ref{eq:B2nd}) can also be 
expressed as 
\begin{equation}
B_{\ell+1}(t)-B_{\ell+1}(0)=
(\ell+1) A_\ell(S)\binom{t+\ell}{\ell+1}. 
\end{equation}
(This formula appeared in \cite[page 209]{wor} as ``The second form of 
Bernoulli function.'')

\section{Background on hyperplane arrangements}
\label{sec:background}

In this section, we recall the proof of the congruence (\ref{eq:congru}) 
presented in \cite{yos-worp}. 

Let $\A=\{H_1, \dots, H_k\}$ be a finite set of affine hyperplanes in 
a vector space $V$. 
We denote the set of all intersections of $\A$ by 
$L(\A)=\{\cap S\mid S\subset \A\}$. 
The set $L(\A)$ is partially ordered by 
reverse inclusion, which has a unique minimal element 
$\hat{0}=V$. 
The characteristic polynomial of $\A$ is 
defined by 
\[
\chi(A, q)=\sum_{X\in L(\A)}\mu(X)q^{\dim X}, 
\]
where $\mu$ is the M\"obius function on $L(\A)$, defined by 
\[
\mu(X)=
\left\{
\begin{array}{ll}
1, & \mbox{ if } X=\hat{0}\\
-\sum_{Y<X}\mu(Y), & \mbox{ otherwise.} 
\end{array}
\right.
\]
Let $V=\{(x_0, x_1, \dots, x_n)\in\R^{\ell+1}\mid \sum x_i=0\}\subset
\R^{\ell+1}$. 
For integers $0\leq i<j\leq \ell$ and $s\in\Z$, 
denote by $H_{ij,s}$ the affine hyperplane 
$\{(x_0, \dots, x_\ell)\in V\mid x_i-x_j=s\}$. 

Let $m\geq 1$ be a positive integer. 
The arrangement 
\[
\scL^m=\{H_{ij,s}\mid 0\leq i<j\leq \ell, 1\leq s\leq m\}
\]
is called the (extended) Linial arrangement (of type $A_\ell$). 
The Linial arrangement has 
several intriguing enumerative properties \cite{ps-def}. 
Postnikov and Stanley \cite{ps-def} (see also \cite{ath-lin}) 
gave the following expression for 
the characteristic polynomial $\chi(\scL^m, t)$. 
\begin{equation}
\label{eq:post}
\chi(\scL^m, t)=
\left(
\frac{1+S+S^2+\dots+S^m}{m+1}
\right)^{\ell+1}t^\ell, 
\end{equation}
where $S$ acts on a function $f(t)$ by $Sf(t)=f(t-1)$ (naturally 
$S^kf(t)=f(t-k)$) as the shift operator. 
Using the Worpitzky identity (\ref{eq:worp}), (\ref{eq:post}) can be 
written as 
\begin{equation}
\label{eq:post2}
\chi(\scL^m, t)=
\left(
\frac{1+S+S^2+\dots+S^m}{m+1}
\right)^{\ell+1}A_\ell(S)
\binom{t+\ell}{\ell}. 
\end{equation}
On the other hand, using the lattice points 
interpretation of the Worpitzky identity, the following formula was obtained 
in \cite{yos-worp} 
\begin{equation}
\label{eq:yoswor}
\chi(\scL^m, t)=
A_\ell(S^{m+1})
\binom{t+\ell}{\ell}. 
\end{equation}
The formulas (\ref{eq:post2}) and (\ref{eq:yoswor}) imply that 
the operator 
\begin{equation}
\label{eq:divis}
\left(
\frac{1+S+S^2+\dots+S^m}{m+1}
\right)^{\ell+1}A_\ell(S)
-
A_\ell(S^{m+1})
\end{equation}
annihilates the polynomial 
$
\binom{t+\ell}{\ell}
$
of degree $\ell$, which means that (\ref{eq:divis}) is divisible by 
$(S-1)^{\ell+1}$ 
(see \cite[Prop. 2.8]{yos-worp}). Hence the congruence 
(\ref{eq:congru}) follows. 

\section{Direct proof of the congruence}
\label{sec:direct}

\subsection{Special case: $m=2$}

We first handle the case $m=2$. 
By considering $F_\ell(x)+F_\ell(-x)$, 
it is easily seen that 
the formal power series $F_\ell(x)=\sum_{k=1}^\infty k^\ell x^k$ satisfies 
\begin{equation}
\label{eq:cancel}
F_\ell(x)-2^{\ell+1}F_\ell(x^2)=-F_\ell(-x). 
\end{equation}
Using the Eulerian polynomial, (\ref{eq:cancel}) can be written as 
\begin{equation}
\label{eq:cancelEulerian}
\left(1+x \right)^{\ell+1}\cdot A_\ell(x)-
2^{\ell+1}\cdot A_\ell(x^2)=
-
\left(1-x \right)^{\ell+1}\cdot A_\ell(-x), 
\end{equation}
which implies the congruence (\ref{eq:congru}) for $m=2$. 

\begin{remark}
Substituting formally $x=1$ into (\ref{eq:cancel}), 
we obtain the formula 
``$F_\ell(1)=\frac{A_\ell(-1)}{2^{\ell+1}(2^{\ell+1}-1)}$.'' 
Then, (\ref{eq:Bnumber}) implies 
``$F_\ell(1)=-\frac{B_{\ell+1}(0)}{\ell+1}$'', which gives  the correct 
value $\zeta(-\ell)=-\frac{B_{\ell+1}(0)}{\ell+1}$ of the Riemann zeta 
function for $\ell\geq 1$. 
\end{remark}

\subsection{General case}

Let $m\geq 2$. 
Denote by $\zeta_m=e^{2\pi\sqrt{-1}/m}$ the primitive $m$-th root of unity. 
We will use the following fact 
\begin{equation}
\label{eq:cycl}
\sum_{j=1}^{m-1}\zeta_m^{jk}=
\left\{
\begin{array}{cl}
m-1,&\mbox{ if $m|k$,}\\
-1,&\mbox{ if $m\not |k$,} 
\end{array}
\right.
\end{equation}
for $k\in\Z$. 

Using definition (\ref{eq:euleriandef}) (or (\ref{eq:poly})), 
the polynomial 
$A_\ell (x^m)-
\left(
\frac{1+x+x^2+\dots+x^{m-1}}{m}
\right)^{\ell +1}
A_\ell (x) $
can be expressed as 
\begin{align*}
&(1-x^m)^{\ell+1}\sum_{k=1}^\infty k^\ell x^{mk}-
\left(
\frac{1+x+x^2+\dots+x^{m-1}}{m}
\right)^{\ell +1}(1-x)^{\ell+1}\sum_{k=1}^\infty k^\ell x^k
\\
&\qquad = 
\left(\frac{1-x^m}{m}\right)^{\ell+1}\cdot 
\left\{m\cdot\sum_{k=1}^\infty(mk)^\ell x^{mk}-\sum_{k=1}^\infty k^\ell x^k
\right\}. 
\end{align*}
It is enough to show that 
\begin{equation}
\label{eq:tobepoly}
P(x):=(1+x+\cdots+x^{m-1})^{\ell+1}
\left\{m\cdot\sum_{k=1}^\infty(mk)^\ell x^{mk}-\sum_{k=1}^\infty k^\ell x^k
\right\} 
\end{equation}
becomes a polynomial. Applying (\ref{eq:cycl}), we have 
\[
\begin{split}
P(x)
&=(1+x+\cdots+x^{m-1})^{\ell+1}
\sum_{k=1}^\infty\sum_{j=1}^{m-1}\zeta_m^{jk}k^\ell x^k
\\
&=
\prod_{i=1}^{m-1}(1-\zeta_m^i x)^{\ell+1}\cdot
\sum_{k=1}^\infty\sum_{j=1}^{m-1}k^\ell (\zeta_m^{j}x)^k\\
&=
\sum_{j=1}^{m-1}
\left(
\prod_{\substack{1\leq i\leq m-1\\i\neq j}}(1-\zeta_m^i x)^{\ell+1}
\right)\cdot
A_\ell(\zeta_m^jx), 
\end{split}
\]
which is a polynomial in $x$. 
This completes the proof of Theorem \ref{thm:main1intro}. 

\begin{remark}
The congruence (\ref{eq:congru}) is not optimal when $\ell$ is even. 
Indeed, if $\ell$ is even, the congruence (\ref{eq:congru}) holds 
modulo $(1-x)^{\ell+2}$, which follows from the symmetry 
$A(\ell, k)=A(\ell, \ell+1-k)$ and $A_\ell(-1)=0$. 
\end{remark}

\section{A characterization of the Eulerian polynomial}

\label{sec:char}

In this section, we prove the following. 

\begin{theorem}
\label{thm:main2proof}
Let $f(x)=x^\ell+a_1x^{\ell-1}+\cdots+a_\ell\in\C[x]$ be a monic 
complex polynomial of degree $\ell>0$. Then, the following are equivalent. 
\begin{itemize}
\item[$(a)$] 
$f(x)=A_\ell(x)$. 
\item[$(b)$] 
For any $m\geq 2$, $f(x)$ satisfies the congruence (\ref{eq:congru}). 
Namely, 
\begin{equation}
\label{eq:congruf}
f(x^m)\equiv
\left(
\frac{1+x+\cdots+x^{m-1}}{m}
\right)^{\ell+1}f(x)\mod (1-x)^{\ell+1}
\end{equation}
is satisfied. 
\item[$(c)$] 
The congruence for $m=2$ holds, namely, 
\[
f(x^2)\equiv
\left(
\frac{1+x}{2}
\right)^{\ell+1}f(x)\mod (1-x)^{\ell+1}
\]
is satisfied. 
\item[$(d)$] 
There exists an integer $m\geq 2$ such that the congruence 
(\ref{eq:congruf}) holds. 
\end{itemize}
\end{theorem}
\begin{proof}
$(a)\Longrightarrow (b)$ is nothing but Theorem \ref{thm:main1intro}. 
The implications $(b)\Longrightarrow (c)\Longrightarrow (d)$ are obvious. 

Let us assume $(d)$. We shall prove $(a)$. Choose an integer $m\geq 2$ 
such that (\ref{eq:congruf}) is satisfied. There exists a polynomial 
$g(x)\in\C[x]$ that satisfies 
\begin{equation}
\label{eq:g}
f(x^m)-
\left(
\frac{1+x+\cdots+x^{m-1}}{m}
\right)^{\ell+1}\cdot
f(x)=(1-x)^{\ell+1}\cdot
g(x). 
\end{equation}
Note that $\deg g=m\ell+m-\ell-2<(m-1)(\ell+1)$. Dividing this equation 
by $(1-x^m)^{\ell+1}$, we have 
\begin{equation}
\label{eq:gfrac}
\frac{f(x^m)}{(1-x^m)^{\ell+1}}-
\frac{1}{m^{\ell+1}}\cdot
\frac{f(x)}{(1-x)^{\ell+1}}=
\frac{g(x)}{(1+x+\cdots+x^{m-1})^{\ell+1}}. 
\end{equation}
We expand $\frac{g(x)}{(1+\cdots+x^{m-1})^{\ell+1}}$ into partial fraction, 
\begin{equation}
\label{eq:partialfrac}
\frac{g(x)}{(1+\cdots+x^{m-1})^{\ell+1}}=
\sum_{j=1}^{m-1}\frac{R_j(x)}{(1-\zeta_m^j x)^{\ell+1}}, 
\end{equation}
where $R_j(x)\in\C[x]$ with $\deg R_j(x)\leq\ell$. 
As is well known in the theory of formal power series \cite{st-ec1}, 
there exist polynomials $\alpha(t), \beta_j(t)\in\C[t]$ 
with $\deg\alpha(t), \deg\beta_j(t)\leq \ell$ 
such that 
\begin{equation}
\label{eq:Rj}
\frac{f(x)}{(1-x)^{\ell+1}}=\sum_{k=0}^\infty\alpha(k)x^k, \mbox{ and }
\frac{R_j(\zeta_m^{-j}x)}{(1-x)^{\ell+1}}=\sum_{k=0}^\infty\beta_j(k)x^k. 
\end{equation}
Then, the right-hand side of (\ref{eq:gfrac}) is 
\begin{equation}
\label{eq:rhs}
\begin{split}
\mbox{RHS of (\ref{eq:gfrac})}
&=
\sum_{j=1}^{m-1}\sum_{k\geq 0}\beta_j(k)\zeta_m^{jk}x^k
\\
&=
\sum_{j=1}^{m-1}\sum_{r=0}^{m-1}\sum_{q=0}^\infty
\beta_j(qm+r)\zeta_m^{j(qm+r)}x^{qm+r}\\
&=
\sum_{q=0}^\infty
\sum_{r=0}^{m-1}
\left(
\sum_{j=1}^{m-1}
\zeta_m^{jr}\beta_j(qm+r)\right)x^{qm+r}. 
\end{split}
\end{equation}
On the other hand, 
the left-hand side of (\ref{eq:gfrac}) is 
\begin{equation}
\label{eq:lhs}
\begin{split}
\mbox{LHS of (\ref{eq:gfrac})}
&=
\sum_{k\geq 0}\alpha(k)x^{km}-\frac{1}{m^{\ell+1}}
\sum_{k\geq 0}\alpha(k)x^{k}
\\
&=
\sum_{q=0}^\infty
\left(\alpha(q)-\frac{1}{m^{\ell+1}}\alpha(mq)\right)x^{mq}-
\sum_{q=0}^\infty\sum_{r=1}^{m-1}\frac{\alpha(mq+r)}{m^{\ell+1}}x^{mq+r}. 
\end{split}
\end{equation}
Comparison of (\ref{eq:rhs}) and (\ref{eq:lhs}) gives 
\begin{equation}
\label{eq:qmr}
\begin{split}
\sum_{j=1}^{m-1}\beta_j(qm)
&=\alpha(q)-\frac{1}{m^{\ell+1}}\alpha(mq)\\
\sum_{j=1}^{m-1}\beta_j(qm+1)\zeta_m^j
&=-\frac{1}{m^{\ell+1}}\alpha(mq+1)\\
&\vdots\\
\sum_{j=1}^{m-1}\beta_j(qm+m-1)\zeta_m^{j(m-1)}
&=-\frac{1}{m^{\ell+1}}\alpha(mq+m-1), 
\end{split}
\end{equation}
for any $q\geq 0$. Since both sides of (\ref{eq:qmr}) are polynomials 
in $q$, we have the following polynomial identities. 
\begin{equation}
\label{eq:polyid}
\begin{split}
\sum_{j=1}^{m-1}\beta_j(t)
&=\alpha\left(\frac{t}{m}\right)-\frac{1}{m^{\ell+1}}\alpha(t)\\
\sum_{j=1}^{m-1}\beta_j(t)\zeta_m^j
&=-\frac{1}{m^{\ell+1}}\alpha(t)\\
&\vdots\\
\sum_{j=1}^{m-1}\beta_j(t)\zeta_m^{j(m-1)}
&=-\frac{1}{m^{\ell+1}}\alpha(t). 
\end{split}
\end{equation}
By summing up all identities in (\ref{eq:polyid}), we obtain a 
functional equation 
\[
\alpha\left(\frac{t}{m}\right)=\frac{1}{m^\ell}\alpha(t). 
\]
This relation is satisfied only by the polynomial of the form 
$\alpha(t)=c_0\cdot t^\ell$, where $c_0\in\C$. Again comparing 
(\ref{eq:euleriandef}) and (\ref{eq:Rj}), $f(x)=c_0\cdot A_\ell(x)$. 
Since $f(x)$ is a monic, we have $f(x)=A_\ell(x)$. 
\end{proof}

\medskip

\noindent
{\bf Acknowledgements.} 
The authors thank Mr. Rei Yoshida for several discussions about 
the contents of this paper. 
The authors also thank Prof. Richard Stanley for noticing \cite{st-za}. 
M. Y. was partially supported by 
JSPS KAKENHI Grant Number 
JP25400060, 
JP15KK0144, 
and 
JP16K13741. 

\end{document}